\documentclass[10pt]{amsart}
 
\usepackage{amsthm,amsmath,amssymb}
\usepackage{graphicx}
\graphicspath{}
\RequirePackage{hyperref}
\RequirePackage{hypernat}

\newtheorem{theorem}{Theorem}

\newtheorem{lemma}[theorem]{Lemma} 
\newtheorem{proposition}[theorem]{Proposition}
\newtheorem{corollary}[theorem]{Corollary}

\theoremstyle{definition} 
\newtheorem{definition}{Definition}
\theoremstyle{remark} 
\newtheorem{example}{Example}

\numberwithin{equation}{section}
\numberwithin{theorem}{section}
\numberwithin{example}{section}
\numberwithin{definition}{section}
\numberwithin{figure}{section}

\newcommand{\an}{{\rm An}}       %for ancestors
\newcommand{\pa}{{\rm pa}}       % for parents
\newcommand{\lhs}[1]{\mathrm{Left}\left({#1}\right)}
\newcommand{\rhs}[1]{\mathrm{Right}\left({#1}\right)}
\newcommand{\tp}[1]{\mathrm{Top}\left({#1}\right)}
\newcommand{\tops}{\mathrm{Tops}}

\newcommand{\bi}{\leftrightarrow}

\makeatletter
\def\@strippedMR{}
\def\@scanforMR#1#2#3\endscan{%
  \ifx#1M\ifx#2R\def\@strippedMR{#3}%
  \else\def\@strippedMR{#1#2#3}%
  \fi\fi}
\renewcommand\MR[1]{\relax
  \ifhmode\unskip\spacefactor3000 \space\fi
  \@scanforMR#1\endscan
  MR\MRhref{\@strippedMR}{\@strippedMR}}
\makeatother

\title[Causal interpretation of mixed graphs]{On the causal
  interpretation of acyclic mixed graphs under multivariate normality}

\author[C.~Fox]{Christopher J. Fox} 
\address{Department of Statistics, The University of Chicago, Chicago,
  IL, U.S.A.}
\email{chrisfox@uchicago.edu}

\author[A.~K\"{a}ufl]{Andreas K\"{a}ufl} 
\address{Institute for Mathematics, University of Augsburg, Augsburg, Germany}
\email{andreas.kaeufl@googlemail.com}

\author[M.~Drton]{Mathias Drton} 
\address{Department of Statistics, University of Washington, Seattle,
  WA, U.S.A.}
\email{md5@uw.edu}

\begin{document}
\begin{abstract}
  In multivariate statistics, acyclic mixed graphs with directed and
  bidirected edges are widely used for compact representation of
  dependence structures that can arise in the presence of hidden
  (i.e., latent or unobserved) variables.  Indeed, under multivariate
  normality, every mixed graph corresponds to a set of covariance
  matrices that contains as a full-dimensional subset the covariance
  matrices associated with a causally interpretable acyclic digraph.
  This digraph generally has some of its nodes corresponding to hidden
  variables.  We seek to clarify for which mixed graphs there exists
  an acyclic digraph whose hidden variable model coincides with the
  mixed graph model.  Restricting to the tractable setting of chain
  graphs and multivariate normality, we show that decomposability of
  the bidirected part of the chain graph is necessary and sufficient
  for equality between the mixed graph model and some hidden variable
  model given by an acyclic digraph.
\end{abstract}


\maketitle

\section{Introduction}
\label{sec:introduction}

Acyclic digraphs are standard representations of causally
interpretable statistical models in which the involved random
variables are noisy functions of each other, with all noise terms
being independent.  In generalization, acyclic mixed graphs with
directed and bidirected edges are widely used for compact
representation of causal structure when important variables are hidden
(that is, unobserved)
\cite{pearl:2009,spirtes:2000,koster:2002,richardson:2002,
  wermuth:2011}.  Such mixed graphs are also known as path diagrams in
the field of structural equation modeling
\cite{bollen:1989,koster:1999}.  The graphs provide, in particular, a
framework for statistical model selection in the presence of hidden
variables \cite{colombo:2012,spirtes:2000,silva:2009}.

Under joint multivariate normality, it is well-known that for every
acyclic mixed graph there exists an acyclic digraph, generally with
additional vertices, such that the statistical model associated with
the digraph is a full-dimensional subset of the model determined by
the mixed graph.  Here, nodes that appear in the digraph but not in
the mixed graph are treated as hidden variables and marginalized over.
In this paper we ask which mixed graphs induce a statistical model
that is not only a superset but equal to a hidden variable model given
by some acyclic digraph.  We focus on the particularly tractable class
of chain graphs, that is, mixed graphs without semi-directed cycles.
Our main result characterizes the chain graphs for which there exists
an acyclic digraph with hidden variable model equal to the chain graph
model.  We begin by formally introducing the concerned graphical
models and stating the precise form of the problem and main result.

Let $\mathcal{D}=(V,E)$ be an acyclic digraph with finite vertex set
$V$ and edge set $E\subseteq V\times V$.  We denote possible edges
$(u,v)$ by $u\to v$.  Let $\mathbb{R}^E$ be the set of matrices
$\Lambda=(\lambda_{uv})\in\mathbb{R}^{V\times V}$ that are supported
on $E$, that is, $\lambda_{uv}=0$ if $u\to v\not\in E$ or $u=v$.  For
any $\Lambda\in\mathbb{R}^E$, the matrix $I-\Lambda$ is invertible
because $\det(I-\Lambda)=1$.  Throughout, $I$ denotes the
identity matrix with size determined by the context.

\begin{definition}
  \label{def:dag-model}
  The \emph{Gaussian graphical model} $\mathbf{N}(\mathcal{D})$ is the
  family of all multivariate normal distributions
  $\mathcal{N}(\mu,\Sigma)$ on $\mathbb{R}^V$ that have covariance
  matrix
  \begin{equation}
    \label{eq:cov-mx}
    \Sigma = (I-\Lambda)^{-T}\Omega(I-\Lambda)^{-1}
  \end{equation}
  with $\Lambda\in\mathbb{R}^E$ and $\Omega\in\mathbb{R}^{V\times V}$
  diagonal with positive diagonal entries.
\end{definition}

The motivation for consideration of the model
$\mathbf{N}(\mathcal{D})$ becomes clearer through the following
construction.  Let $\epsilon=(\epsilon_v)_{v\in V}$ be a multivariate
normal random vector with covariance matrix $\Omega=(\omega_{uv})$,
and let $\pa(v)=\{u\,:\, u\to v\in E\}$ be the set of \emph{parents}
of vertex $v$.  Define the random vector $X=(X_v)_{v\in V}$ to be the
solution of the linear equation system
\begin{equation}
  \label{eq:sem}
  X_v = \lambda_{0v}+ \sum_{u\in\pa(v)} \lambda_{uv} X_u+\epsilon_v,
  \qquad v\in V.
\end{equation}
Then $X$ is multivariate normal with covariance matrix as
in~(\ref{eq:cov-mx}), where $\Lambda\in\mathbb{R}^E$ contains the
coefficients in~the equations in (\ref{eq:sem}), which are also known
as \emph{structural equations} \cite{bollen:1989}.  The distributions
in $\mathbf{N}(\mathcal{D})$ thus arise when the considered random
variables are related via noisy functional, or in other words, causal
relationships.  We refer the reader to \cite{lauritzen:1996} or
\cite[Chap.~3]{oberwolfach} for general background on graphical
models, and to \cite{pearl:2009} and \cite{spirtes:2000} for details
on causal interpretation.

While noisy functional relationships are often natural for modelling
dependences among observed random variables $X_v$, $v\in V$, many
applications face the problem that additional variables may
appear in the functions.  The relevant acylic digraph
$\mathcal{D}=(U,E)$ has then a larger vertex set $U\supseteq V$ and
edge set $E\subseteq U\times U$.  The nodes in $U\setminus V$
correspond to so-called hidden (or latent) variables that remain
unobserved.  The pair $(\mathcal{D},V)$ determines a \emph{hidden
  variable model} $\mathbf{N}_V(\mathcal{D})$ comprising the
normal distributions on $\mathbb{R}^V$ that arise as $V$-marginal of
a distribution in $\mathbf{N}(\mathcal{D})$.

\begin{figure}[t]
  \centering
  (a) %\hspace{.1cm}
  \includegraphics[scale=0.25]{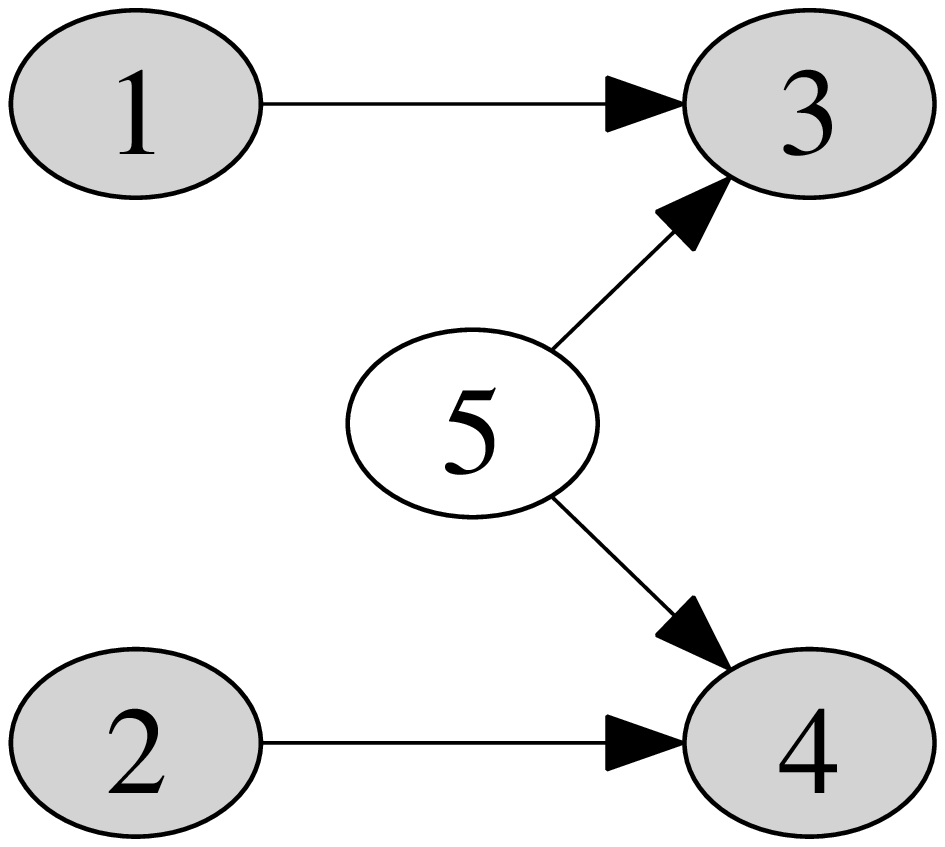} 
  \hspace{1cm}
  (b) %\hspace{.1cm}
  \includegraphics[scale=0.25]{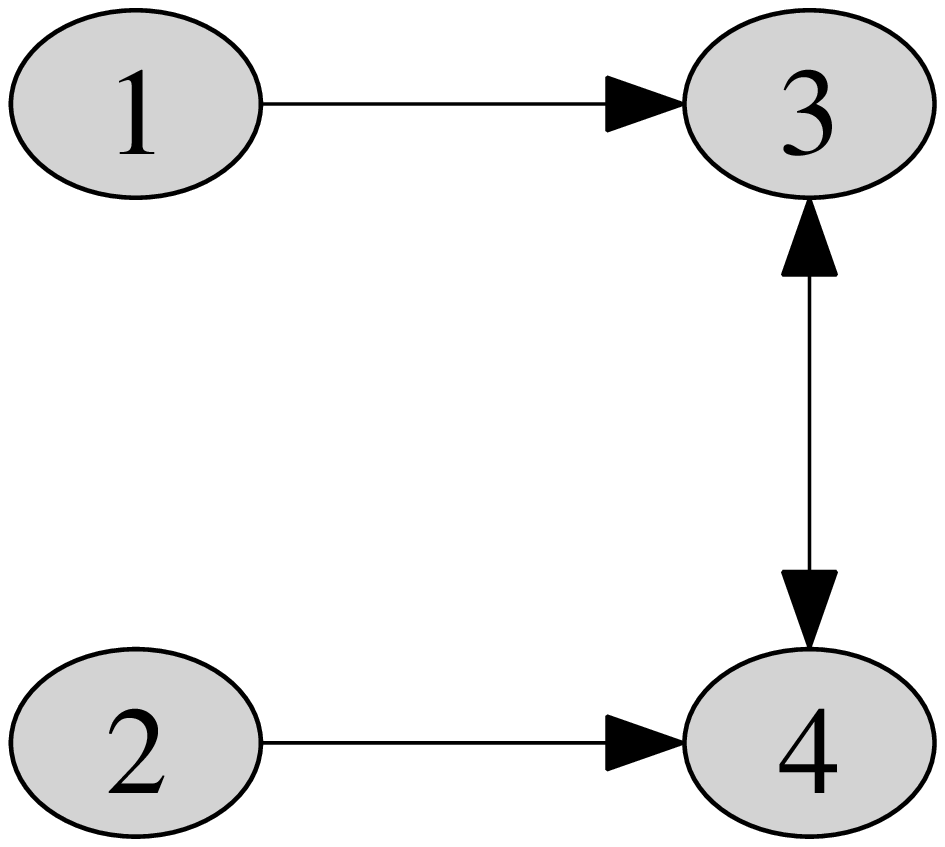}
  \caption{(a) An acyclic digraph $\mathcal{D}$.  (b)
    Mixed graph for the marginal distribution of $(X_1,\dots,X_4)$.}
  \label{fig:sur}
\end{figure}

\begin{example}
  Consider the acyclic digraph $\mathcal{D}=(U,E)$ from
  Figure~\ref{fig:sur}(a) with vertex set $U= \{1,\dots,5\}$.  If we
  only observe the random variables indexed by the nodes in
  $V=\{1,\dots,4\}$, then the covariance matrices of the normal
  distributions in $\mathbf{N}_V(\mathcal{D})$ have the form
  \begin{equation}
    \label{eq:sur_cov_mx_hidden}
    \begin{pmatrix}
      \omega_{11} & 0 & \lambda_{13}\omega_{11} & 0 \\
      0 & \omega_{22} & 0 & \lambda_{24}\omega_{22} \\
      \lambda_{13}\omega_{11} & 0
      &\lambda_{13}^2\omega_{11}+\omega_{33}+\lambda_{53}^2\omega_{55}
      & \lambda_{53}\lambda_{54}\omega_{55} \\ 
      0 & \lambda_{24}\omega_{22} &
      \lambda_{53}\lambda_{54}\omega_{55}&
      \lambda_{24}^2\omega_{22}+\omega_{44}+\lambda_{54}^2\omega_{55}
    \end{pmatrix},
  \end{equation}
  where $\omega_{11},\dots,\omega_{55}>0$ are the variances of the
  error terms in (\ref{eq:sem}).  The four edges in $\mathcal{D}$
  correspond to the coefficients
  $\lambda_{13},\lambda_{24},\lambda_{53},\lambda_{54}\in\mathbb{R}$.
\end{example}

Many acylic digraphs give the same hidden variable model over a
particular set of observed nodes $V$.  For instance, if we add a sixth
node and the edges $3\leftarrow 6 \to 4$ to the graph $\mathcal{D}$
from Figure~\ref{fig:sur}(a), then the resulting graph $\mathcal{D}'$
satisfies $\mathbf{N}_V(\mathcal{D}')=\mathbf{N}_V(\mathcal{D})$ for
$V=\{1,\dots,4\}$.  Due to this fact, it is often useful to represent
hidden variable models by mixed graphs whose nodes correspond only to
observed variables but whose edge set may contain a second type of
edge.

Mixed graphs are triples $\mathcal{G}=(V,D,B)$ that consist of a
finite vertex set $V$ and two sets of edges $D,B\subseteq V\times V$.
The set $D$ contains directed edges, denoted again by $u\to
v$.  The pairs in $B$ are bidirected edges, which we write as $v\bi
w$.  The bidirected edges do not have an orientiation, that is, $v\bi
w\in B$ if and only if $w\bi v\in B$.  We tacitly assume absence of
self-loops, that is, $(v, v)\not\in D\cup B$ for all $v\in V$.  A
mixed graph is \emph{acyclic} if its \emph{directed part} $(V,D)$ is
an acyclic digraph.

For an acyclic mixed graph $\mathcal{G}=(V,D,B)$, let $\mathbb{R}^D$
be the set of $V\times V$-matrices $\Lambda=(\lambda_{uv})$ supported
on $D$.  Similarly, let $\mathit{PD}(B)$ be the cone of positive
definite $V\times V$-matrices $\Omega=(\omega_{uv})$ supported on $B$,
with the distinction that the diagonal entries are not constrained to
be zero.  Mixed graph models are defined in analogy to the models
given by digraphs, the sole difference being the fact that the error
terms in (\ref{eq:sem}) need no longer be mutually independent.

\begin{definition}
  \label{def:mixed-graph-model}
  The \emph{Gaussian mixed graph model} $\mathbf{N}(\mathcal{G})$ is
  the family of all multivariate normal distributions
  $\mathcal{N}(\mu,\Sigma)$ on $\mathbb{R}^V$ that have covariance
  matrix
  \begin{equation*}
    \label{eq:cov-mx-mixed}
    \Sigma = (I-\Lambda)^{-T}\Omega(I-\Lambda)^{-1}
  \end{equation*}
  with $\Lambda\in\mathbb{R}^D$ and $\Omega\in\mathit{PD}(B)$.
\end{definition}

A mixed graph model does not have an \emph{a priori} causal
interpretation because no causal mechanism is specified for generating
the dependences that may exist among the error terms $\epsilon_v$,
$v\in V$.  However, as we suggested above, such a causal
interpretation can be given via hidden variables.  

\begin{example}
  The distributions in the model $\mathbf{N}(\mathcal{G})$ defined by
  the mixed graph $\mathcal{G}$ in Figure~\ref{fig:sur}(b) have
  covariance matrix
  \begin{equation}
    \label{eq:sur_cov_mx}
    \begin{pmatrix}
      \omega_{11} & 0 & \lambda_{13}\omega_{11} & 0\\
      0 & \omega_{22} & 0 & \lambda_{24}\omega_{22}\\
      \lambda_{13}\omega_{11} & 0 & \omega_{33} +
      \lambda_{13}^2\omega_{11} & \omega_{34}\\
      0 & \lambda_{24}\omega_{22} & \omega_{34} & \omega_{44} +
      \lambda_{24}^2\omega_{22} 
    \end{pmatrix},
  \end{equation}
  where $\lambda_{13},\lambda_{24}\in\mathbb{R}$ and the five
  parameters $\omega_{ij}$ form a positive definite $4\times 4$ matrix
  with all off-diagonal entries zero except for $\omega_{34}$.  It is
  not difficult to show that the matrix in~(\ref{eq:sur_cov_mx}) can
  always be written as in (\ref{eq:sur_cov_mx_hidden}), and vice
  versa.  Hence, $\mathbf{N}(\mathcal{G})=\mathbf{N}_V(\mathcal{D})$
  for the graph $\mathcal{D}$ from Figure~\ref{fig:sur}(a) and the
  nodes in $V=\{1,\dots,4\}$.
\end{example}

The two graphs in Figure~\ref{fig:sur} are related through a
well-known general construction.  For each bidirected edge $u\bi v$ in
a mixed graph $\mathcal{G}$ introduce a new node $h_{\{u,v\}}$.  Then
replace the edge $u\bi v$ by the two edges $u\leftarrow h_{\{u,v\}}
\to v$.  The resulting digraph $\mathcal{D}$ has been called the
bidirected subdivision in \cite{sullivant:2010} and the canonical DAG
(short for acyclic digraph) in \cite{richardson:2002}; it also
underlies the semi-Markovian models of \cite{pearl:2009}.  As used in
\cite{sullivant:2010}, the construction yields for every mixed graph
$\mathcal{G}=(V,D,B)$ a digraph $\mathcal{D}=(U,E)$ with $V\subseteq U$
such that $\mathbf{N}_V(\mathcal{D})$ is a full-dimensional subset of
$\mathbf{N}(\mathcal{G})$; to make this a formal statement, identify
each model with the set of covariance matrices of its probability
distributions.  Hence, every mixed graph model can be regarded as a
``closure'' of the causal (hidden variable) model defined by its
canonical digraph.  

While the mixed graph from Figure~\ref{fig:sur}(b) and its canonical
digraph in Figure~\ref{fig:sur}(a) define exactly the same model for
the random vector $(X_1,\dots,X_4)$, it is known from the examples
in \cite{richardson:2002} and \cite{drtonyu:2010} that the canonical
digraph sometimes only defines a proper submodel.  It remains an open
problem to characterize the acyclic mixed graphs that have the
following strict causal interpretation.

\begin{definition}
  \label{def:strict-causal}
  An acyclic mixed graph $\mathcal{G}=(V,D,B)$ is \emph{strictly
    Gaussian causal} if there exists an acyclic digraph
  $\mathcal{D}=(U,E)$ on $U\supseteq V$ with
  $\mathbf{N}_V(\mathcal{D})=\mathbf{N}(\mathcal{G})$.
\end{definition}

We would like to emphasize that the acyclic digraphs in Definition~\ref{def:strict-causal} are entirely arbitrary.  Hence, unlike canonical DAGs, they may feature hidden nodes in $U\setminus V$ with more than two children.

The following result is known about graphs without directed edges.
Recall that the \emph{bidirected part} $(V,B)$ is decomposable (or
chordal or triangulated) if it has no induced cycles of length more
than three. 

\begin{theorem}[\cite{drtonyu:2010}]
  \label{thm:drtonyu}
  If $\mathcal{G}=(V,\emptyset,B)$ is a mixed graph without directed
  edges, then it is strictly Gaussian causal if $(V,B)$ is
  decomposable.
\end{theorem}

In this paper we first clarify that the hidden variable construction
underlying the proof of Theorem~\ref{thm:drtonyu} also yields that this
condition is sufficient in general.

\begin{theorem}
  \label{thm:main1}
  If an acyclic mixed graph $\mathcal{G}=(V,D,B)$ has a decomposable
  bidirected part $(V,B)$, then $\mathcal{G}$ is strictly Gaussian
  causal.
\end{theorem}

We remark that it is meaningful to extend
Definition~\ref{def:mixed-graph-model} to non-acyclic mixed graphs,
restricting the matrices $\Lambda\in\mathbb{R}^D$ to have $I-\Lambda$
invertible.  Then Theorem~\ref{thm:main1} would still apply if
Definition~\ref{def:strict-causal} were changed to allow for cyclic
digraphs.

In general, the decomposability of the bidirected part is not
necessary for strict Gaussian causality of an acyclic mixed graph; see
Example~\ref{ex:decomp-not-necessary}.  However, it is necessary when
$\mathcal{G}=(V,D,B)$ is a \emph{chain graph}, which refers to a mixed
graph without semi-directed cycles.  An $n$-cycle
$(v_1,v_2),(v_2,v_3),\dots,(v_n,v_1)\in D\cup B$ is semi-directed if
at least one of its edges is in $D$; any directed edge
$(v_i,v_{i+1})\in D$ on this cycle is traversed in the same
orientation $v_i\to v_{i+1}$.   We remark that the statistical interpretation of chain graphs in \cite{lauritzen:1996} differs from the one in this paper; see also \cite{drton:2009,wermuth:2004}.

\begin{theorem}
  \label{thm:main}
  Suppose the mixed graph $\mathcal{G}=(V,D,B)$ is a chain graph.
  Then $\mathcal{G}$ is strictly Gaussian causal if and only if the
  bidirected part $(V,B)$ is decomposable.
\end{theorem}

A chain graph is simple (that is, $D\cap B=\emptyset$), and the
connected components of its bidirected part are also known as
\emph{chain components}.  Each chain
component induces a subgraph that is bidirected, that is, does not
contain any edge in $D$.  Moreover, for a chain graph, the model
$\mathbf{N}(\mathcal{G})$ can also be defined by conditional
independence constraints among the observed random variables.  We use this fact in our proof of
necessity in Theorem~\ref{thm:main}, which also involves a sign-change
trick from \cite{drtonyu:2010} and results on subdeterminants of the
covariance matrix in~(\ref{eq:cov-mx}) due to \cite{sullivant:2010}.
We remark that proving a mixed graph not to be strictly Gaussian
causal requires arguments about an infinite set of acyclic digraphs.
This is in contrast to many other Markov equivalence problems, where all
considered graphs have the same vertex set
\cite{pearl:1994,drton:2008,zhao:2005,ali:2009,wermuth:test}.

The remainder of the paper is organized as follows.  In
Section~\ref{sec:sufficiency}, we prove Theorem~\ref{thm:main1}.
Section~\ref{sec:preliminaries} reviews background needed for the
proof of Theorem~\ref{thm:main}, which is the topic of
%% .  To clarify the
%% main ideas, we begin by treating bidirected cycles in
%% Section~\ref{sec:pcycle} and then show how to extend the arguments to
%% chain graphs in 
Section~\ref{sec:chaingraphs}.  We conclude with a
discussion of the treated problem in Section~\ref{sec:conclusion}.

%%%%%%%%%%%%%%%%%%%%%%%%%%%%%%%%%%%%%%%%%%%%%%%%%

\section{Mixed graphs with decomposable bidirected part}
\label{sec:sufficiency}

In this section we prove Theorem \ref{thm:main1} according to which an
acyclic mixed graph $\mathcal{G}=(V,D,B)$ with decomposable bidirected
part $(V,B)$ is strictly Gaussian causal.  Let $\mathcal{C}$ be the
set of all cliques of the part $(V,B)$, where a set $C\subset V$ is a
clique if $u\bi v\in B$ for any two distinct nodes $u,v\in C$.  Let
$\mathcal{C}_2\subset \mathcal{C}$ be the set of cliques that have two
or more elements.  We will use the following construction.

\begin{definition}
  \label{def:clique-graph}
  We define the clique digraph of a mixed graph $\mathcal{G}=(V,D,B)$
  to be the acyclic digraph $\mathcal{D(G)}=(U,E)$ with $U=V\cup
  \mathcal{C}_2$ and
  \[
  E = D\cup \{ h\to v \,:\, h\in\mathcal{C}_2,\, v\in h\}. 
  \]
\end{definition}

The clique digraph contains a new node for every non-singleton clique
in $(V,B)$, and links each new node to all nodes appearing in the
concerned clique.  Figure~\ref{fig:canonical} shows an example.  If
$(V,B)$ contains only cliques of size at most two, then the clique
digraph is equal to the aforementioned bidirected
subdivision/canonical DAG.

\begin{figure}[t]
  \centering
  (a) \hspace{-0cm}
  \includegraphics[scale=0.25]{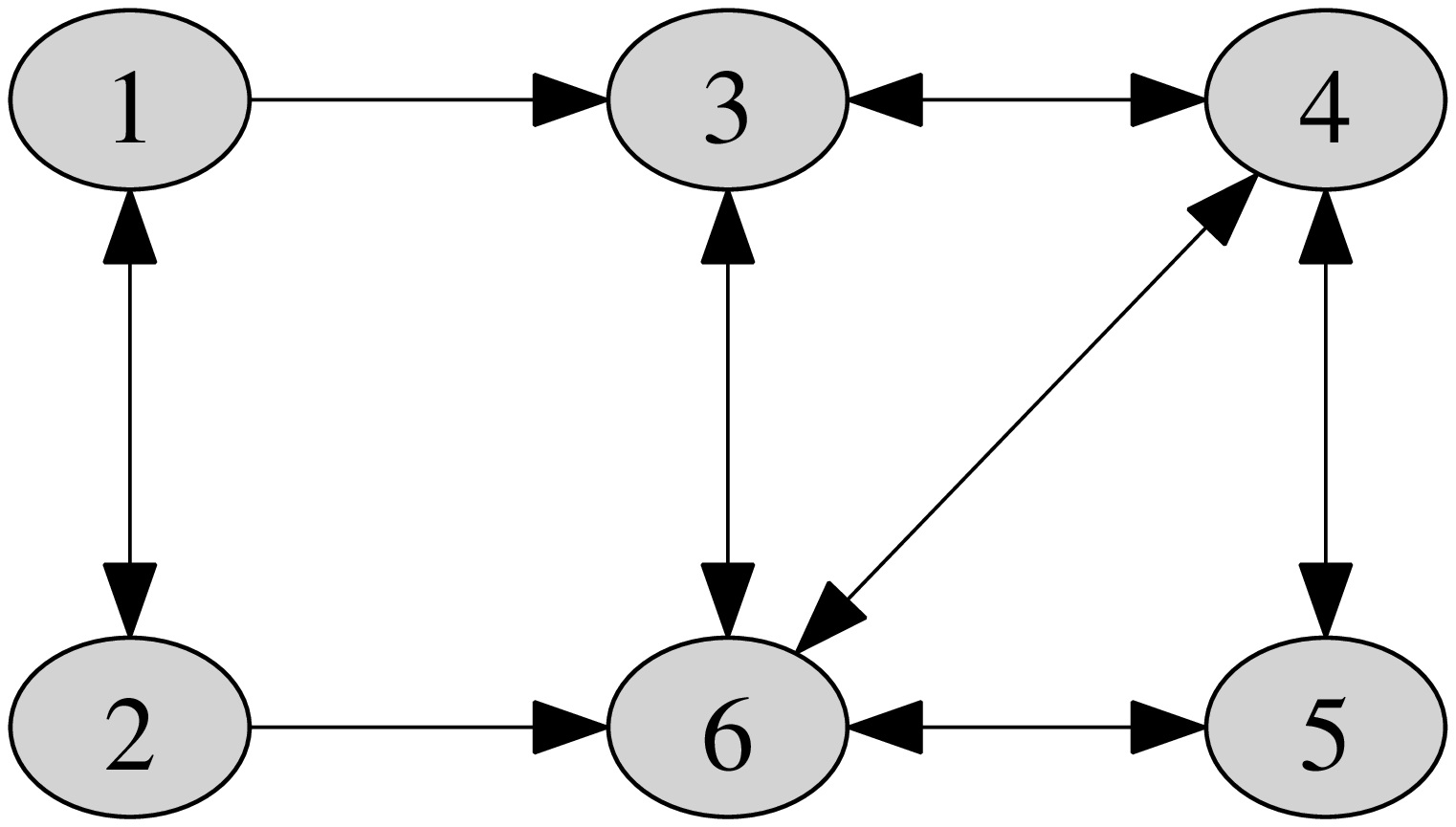} 
  \hspace{0.5cm}
  (b) \hspace{-0cm}
  \includegraphics[scale=0.25]{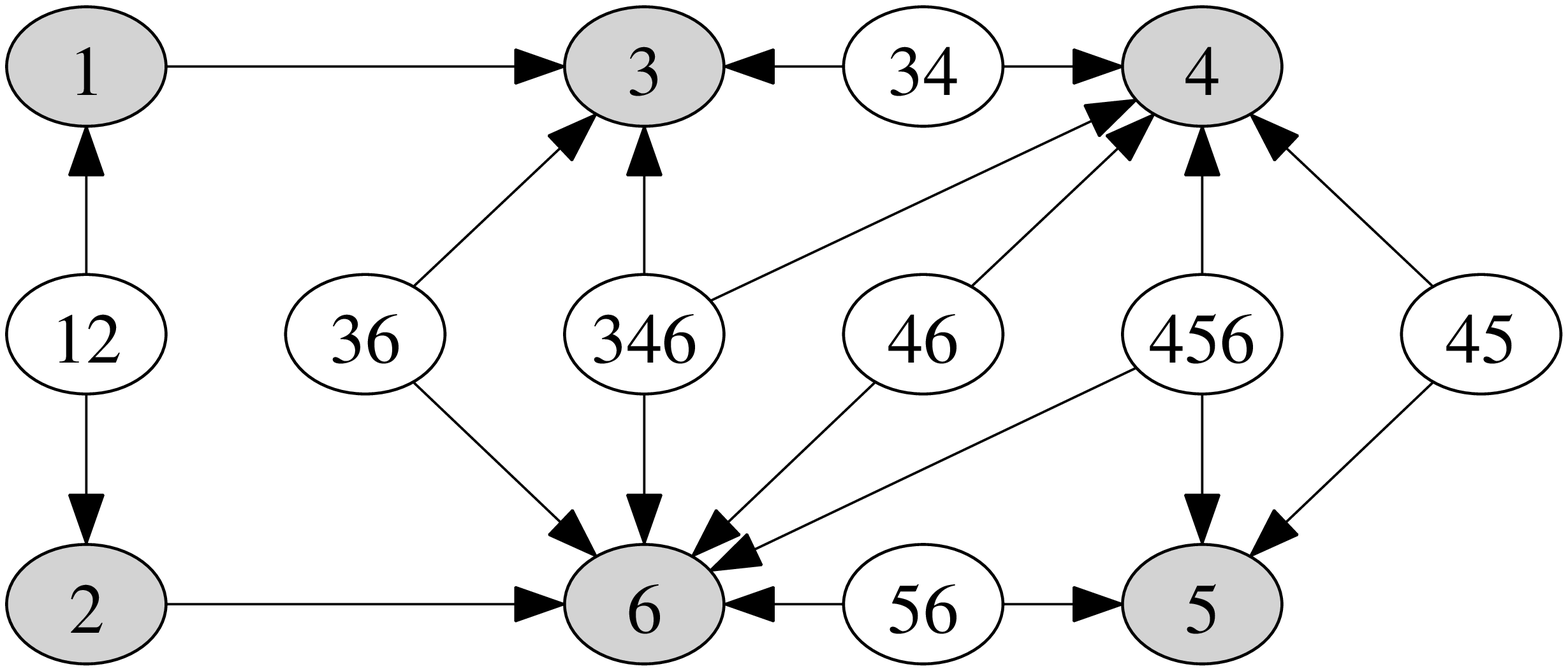}
  \caption{(a) A chain graph $\mathcal{G}$ with decomposable
    bidirected part.  (b) The clique digraph $\mathcal{D(G)}$;
    unshaded nodes correspond to cliques and represent the hidden
    variables.}
  \label{fig:canonical}
\end{figure}

\begin{proof}[Proof of Theorem~\ref{thm:main1}]
  Let $\mathcal{D(G)}=(U,E)$ be the clique digraph of the acyclic
  mixed graph $\mathcal{G}=(V,D,B)$, with $U=V\cup\mathcal{C}_2$.  Let
  $\Gamma$ be a $U\times U$ matrix in $\mathbb{R}^E$, and let $\Delta$
  be a diagonal $U\times U$ matrix with positive diagonal entries.
  Based on the partitioning $U=V\cup\mathcal{C}_2$, we have
  \begin{align*}
    \Gamma&=
    \begin{pmatrix}
      \Gamma_{11} & 0 \\
      \Gamma_{21} & 0 
    \end{pmatrix},
    &
    \Delta&=
    \begin{pmatrix}
      \Delta_{11} & 0 \\
      0 & \Delta_{22}
    \end{pmatrix},
  \end{align*}
  where $\Gamma_{11},\Delta_{11}$ are in $\mathbb{R}^{V\times V}$ and
  $\Delta_{22}$ is in $\mathbb{R}^{\mathcal{C}_2\times \mathcal{C}_2}$.
  Due to the triangular form of $I-\Gamma$, we have
  \begin{equation*}
    (I-\Gamma)^{-1} = 
    \begin{pmatrix}
      (I-\Gamma_{11})^{-1} & 0\\
      \Gamma_{21}(I-\Gamma_{11})^{-1} & I
    \end{pmatrix}.
  \end{equation*}
  The covariance matrices for the distributions in
  $\mathbf{N}_V(\mathcal{D(G)})$ are thus of the form
  \begin{align}
    \nonumber
    \Sigma &= 
    \left[(I-\Gamma)^{-T}\Delta (I-\Gamma)^{-1}\right]_{V\times V}\\
    \label{eq:clique-bidi-part}
    &=(I-\Gamma_{11})^{-T} \left[ 
      \begin{pmatrix}
        I & 0\\
        -\Gamma_{21} & I
      \end{pmatrix}^{-T}
      \Delta
      \begin{pmatrix}
        I & 0\\
        -\Gamma_{21} & I
      \end{pmatrix}^{-1}
    \right]_{V\times V}
    (I-\Gamma_{11})^{-1}\\
    \nonumber
    &=(I-\Gamma_{11})^{-T}
    (\Delta_{11}+\Gamma_{21}^T\Delta_{22}\Gamma_{21})(I-\Gamma_{11})^{-1}. 
  \end{align}

  Since two columns of $\Gamma_{21}$ have disjoint support unless the
  corresponding two nodes are in a clique in $\mathcal{C}_2$, or
  equivalently, unless the two nodes are adjacent in $\mathcal{G}$, the
  matrix
  \begin{equation}
    \label{eq:clique-omega}
    \Delta_{11}+\Gamma_{21}^T\Delta_{22}\Gamma_{21}
  \end{equation}
  is a positive definite matrix in $\mathit{PD}(B)$.  Hence,
  $\mathbf{N}_V(\mathcal{D(G)})\subseteq\mathbf{N}(\mathcal{G})$.
  However, more is true.  According to~(\ref{eq:clique-bidi-part}),
  the matrix in~(\ref{eq:clique-omega}) is a covariance matrix
  associated with the clique digraph of the mixed graph
  $(V,\emptyset,B)$.  The proof of Theorem~\ref{thm:drtonyu} in
  \cite{drtonyu:2010} shows that, for $(V,B)$ decomposable, any matrix
  in $\mathit{PD}(B)$ can be written in the
  form~(\ref{eq:clique-omega}).  We conclude that
  $\mathbf{N}_V(\mathcal{D(G)})=\mathbf{N}(\mathcal{G})$.
\end{proof}

\begin{figure}[t]
  \centering
  (a) \includegraphics[scale=0.25]{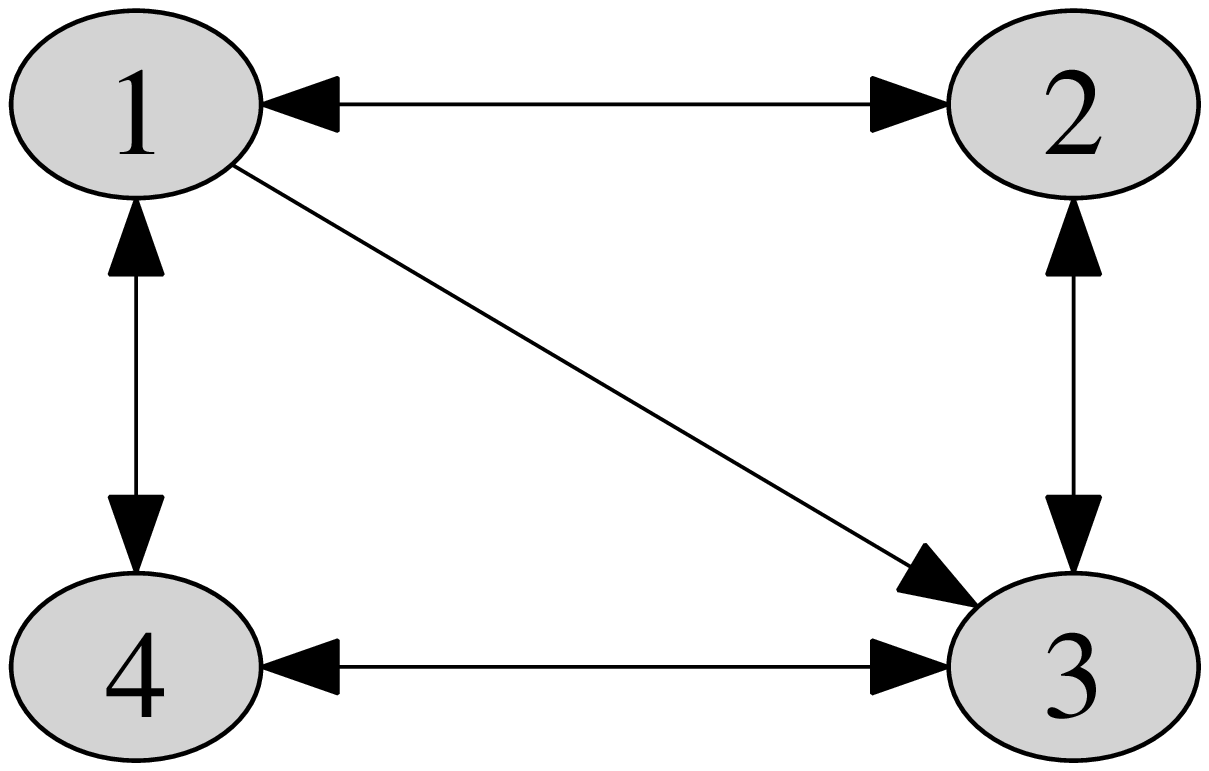}
  \hspace{1cm}
  (b) \includegraphics[scale=0.25]{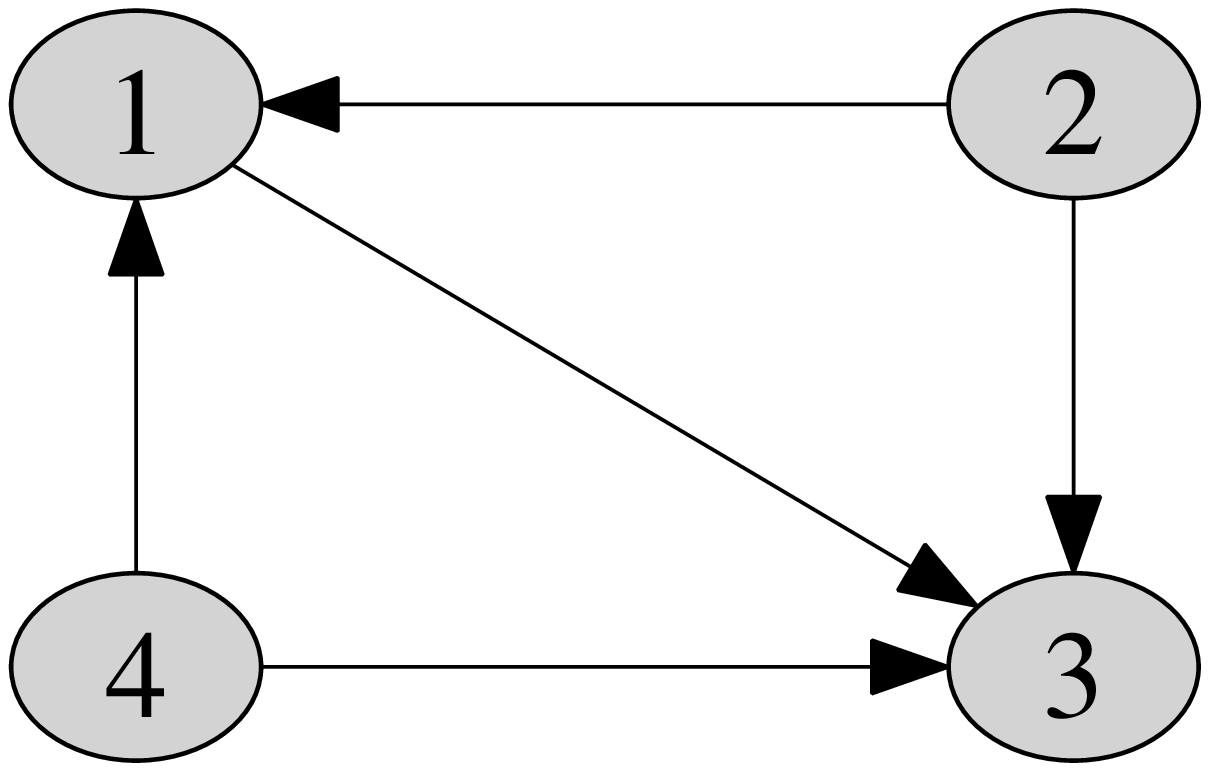}
  \caption{(a) An acyclic mixed graph that has non-decomposable
  bidirected part 
    but is strictly Gaussian causal.  (b) An acyclic digraph
    determining the same model. }
  \label{fig:decomp-not-necessary}
\end{figure}

The next example shows that decomposability of the bidirected part of
an acyclic mixed graph is not necessary for strict Gaussian causality.

\begin{example}
  \label{ex:decomp-not-necessary}
  Let $\mathcal{G}$ be the mixed graph depicted in
  Figure~\ref{fig:decomp-not-necessary}(a).  The bidirected part of
  this graph is a four-cycle, so not decomposable.  However,
  $\mathcal{G}$ is strictly Gaussian causal because
  $\mathbf{N}(\mathcal{G})=\mathbf{N}(\mathcal{D})$ for the acyclic
  digraph $\mathcal{D}$ from Figure~\ref{fig:decomp-not-necessary}(b).
\end{example}

The two graphs $\mathcal{G}$ and $\mathcal{D}$ in
Figure~\ref{fig:decomp-not-necessary} have the same vertex set, and
the fact that $\mathbf{N}(\mathcal{G})=\mathbf{N}(\mathcal{D})$ is an
instance of Markov equivalence of two mixed graphs that are ancestral
in the sense of \cite{richardson:2002}.  More generally, for ancestral
graphs, the sufficient condition from Theorem~\ref{thm:main1} could be
strengthened by first applying results on the characterization of
Markov equivalence of ancestral graph \cite{ali:2009,zhao:2005} to
convert a given ancestral mixed graph $\mathcal{G}$ to another
ancestral mixed graph $\mathcal{G'}$ with $\mathbf{N}(\mathcal{G})=\mathbf{N}(\mathcal{G'})$
with fewer bidirected edges; this is in the spirit of
\cite{drton:2008}.  If the bidirected part of $\mathcal{G'}$ is
decomposable, then Theorem~\ref{thm:main1} can be applied.

%%%%%%%%%%%%%%%%%%%%%%%%%%%%%%%%%%%%%

\section{Treks, systems of treks
  and d-connecting walks}
\label{sec:preliminaries}

For a proof of Theorem~\ref{thm:main}, which is the topic of
Section~\ref{sec:chaingraphs}, we need to be able to make arguments
about the structure of an acyclic digraph $\mathcal{D}$ that
determines a hidden variable model equal to a given chain graph model
$\mathbf{N}(\mathcal{G})$.  In preparation, we collect in this section
known results about the combinatorial structure of the covariance
matrices of distributions in $\mathbf{N}(\mathcal{D})$.

Let $\mathcal{D}=(U,E)$ be any acyclic digraph.  A \emph{walk} $\pi$
from \emph{source} node $u$ to \emph{target} node $v$ in $\mathcal{D}$
is a sequence of edges in $E$ connecting the consecutive nodes in a
sequence of nodes starting at $u$ and ending at $v$.  If $\pi$ visits
all of its nodes only once, then it is a \emph{path}.  If all edges
are traversed according to their orientation, then the walk $\pi$ is a
\emph{directed path} from $u$ to $v$; it is a path because visiting a
node twice would result in a directed cycle.  A \emph{collider} on a
walk $\pi$ from $u$ to $v$ is an interior node $w$ (i.e.,
$w\notin\{u,v\}$) such that the two edges of $\pi$ that are incident
to $w$ have their ``arrowheads collide'' as $w'\to w\leftarrow w''$.

A \emph{trek} $\tau$ from $u$ to $v$ is a walk without colliders and
takes the form:
\begin{equation}
  \label{eq:trek1}
  v^{\text{L}}_{l}\leftarrow
  v^{\text{L}}_{l-1}\leftarrow \dots \leftarrow
  v^{\text{L}}_1\leftarrow
  v^{\text{T}}\to  v^{\text{R}}_1\to
  \dots \to v^{\text{R}}_{r-1}\to v^{\text{R}}_{r},
\end{equation}
where the endpoints are $v^{\text{L}}_{l}=u$, $v^{\text{R}}_{r}=v$.
We say that
$\lhs{\tau}=\{v^{\text{T}},v^{\text{L}}_1,\dots,v^{\text{L}}_{l}\}$ is
the left-hand side of $\tau$, and similarly,
$\rhs{\tau}=\{v^{\text{T}},v^{\text{R}}_1,\dots,v^{\text{R}}_{r}\}$ is
the right-hand side.  The \emph{top} node $v^\text{T}=\tp{\tau}$ is
contained in both sides of the trek.  Though a trek in an acyclic
digraph has no repetition of nodes on either the right- or left-hand
side, it may contain the same node once on its left-hand side and once
again on its right-hand side and thus not be a path.  Every directed
walk is a trek with $|\lhs{\tau}|=1$ or $|\rhs{\tau}|=1$ depending on
the orientation of the edges.  A trek is allowed to be \emph{trivial},
that is, for every node $v\in U$, there is a trek $\tau$ from $v$ to
$v$ that contains no edges and has $\lhs{\tau}=\rhs{\tau}=\{v\}$ and
$\tp{\tau} = v$.
 
\begin{figure}[t]
  \centering
  \includegraphics[scale=0.25]{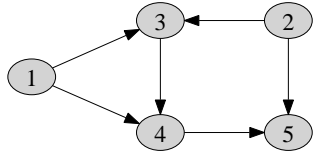}
  \caption{An acyclic digraph.}
  \label{fig:trek_examples}
\end{figure}

\begin{example}
  In the graph shown in Figure~\ref{fig:trek_examples}, the path
  $\pi_1: 3\leftarrow 1\to 4\to 5$ is a trek with
  $\lhs{\pi_1}=\{1,3\}$, $\rhs{\pi_1}=\{1,4,5\}$, and $\tp{\pi_1} =
  1$.  Similarly, the walk $\pi_2: 4\leftarrow 3\leftarrow 1\to 3\to
  4\to 5$ is a trek with $\lhs{\pi_2}=\{1,3,4\}$,
  $\rhs{\pi_2}=\{1,3,4,5\}$, and $\tp{\pi_2} = 1$.  The walk $\pi_3: 3
  \rightarrow 4 \leftarrow 1\to 3$ is not a trek due to the
  collider at node 4.
\end{example}

Let $\Lambda=(\lambda_{uv})\in\mathbb{R}^E$, and let
$\Omega=(\omega_{vv})$ be a diagonal $U\times U$ matrix with positive
diagonal entries.  
If a covariance matrix $\Sigma=(\sigma_{uv})$ satisfies
$\Sigma=(I-\Lambda)^{-T}\Omega(I-\Lambda)^{-1}$ as
in~(\ref{eq:cov-mx}), then the matrix $\Lambda\in\mathbb{R}^E$ and the
diagonal matrix $\Omega$ in this representation are unique.  Indeed,
if $u\to v\in E$, then $\lambda_{uv}$ corresponds to entry $u$ in the vector
\begin{equation}
  \label{eq:regress-coeff}
\Sigma_{v,\pa(v)}\left(\Sigma_{\pa(v),\pa(v)}\right)^{-1} 
\end{equation}
and $\omega_{vv}$ is a Schur complement, namely,
\begin{equation}
  \label{eq:cond-var}
  \omega_{vv}=\sigma_{vv} - 
  \Sigma_{v,\pa(v)}\left(\Sigma_{\pa(v),\pa(v)}\right)^{-1}
  \Sigma_{\pa(v),v}. 
\end{equation}
For a general discussion of this uniqueness see \cite{drton:2011}.

For a trek $\tau$ in $\mathcal{D}$ with
$\tp{\tau}=t$, define the trek monomial
\begin{equation}
  \label{eq:trekmonomial1}
  \sigma(\tau) = \omega_{tt}\prod_{x\to y\in \tau} \lambda_{xy}.
\end{equation}
The unique representation implies that the value of the trek monomial
$\sigma(\tau)$ is determined by $\Sigma$ via~(\ref{eq:regress-coeff})
and~(\ref{eq:cond-var}).  The following rule expresses the entries of
a covariance matrix
%% $\Sigma$ for $\mathbf{N}(\mathcal{D})$ 
%% =(I-\Lambda)^{-T}\Omega(I-\Lambda)^{-1}$
%% from~(\ref{eq:cov-mx}) 
as sums of trek
monomials~\cite{spirtes:2000,wright:1921,wright:1934}.

\begin{lemma}[Trek rule]
  \label{lem:trek-rule}
  The covariance matrix $\Sigma=(\sigma_{uv})$ of a distribution from
  $\mathbf{N}(\mathcal{D})$ has the entries
  \begin{equation}
    \label{eq:trek-rule}
    \sigma_{uv} = \sum_{\tau\in\mathcal{T}(u,v)} \sigma(\tau),
  \end{equation}
  where $\mathcal{T}(u,v)$ is the set of all treks from $u$ to $v$,
  which is finite.
\end{lemma}
 
In the treatment of chain graphs we will have information about
subdeterminants rather than entries of the covariance matrix.  This will
require us to consider sets of treks $\Pi=\{\tau_1,\dots,\tau_n\}$.
Let $x_i$ and $y_i$ be the source and the target of each trek
$\tau_i$, respectively.  If all sources and all targets are distinct,
then we call $\Pi$ a \emph{system of treks} from $X=\{x_1,\dots,x_n\}$
to $Y=\{y_1,\dots,y_n\}$, denoted as $\Pi:X\rightrightarrows Y$.  This
allows $X\cap Y\not=\emptyset$.  If
\[
\lhs{\tau_i}\cap\lhs{\tau_j}=\emptyset=\rhs{\tau_i}\cap\rhs{\tau_j}
\]
for all $i\neq j$, then we say that the system of treks
$\Pi=\{\tau_1,\dots,\tau_n\}$ has \emph{no sided intersection}.  We
then have the following generalization of the trek rule
\cite{sullivant:2010}.

\begin{lemma}[\cite{sullivant:2010}]
  \label{lem:std}
  Suppose $\Sigma=(\sigma_{uv})$ is the covariance matrix of a
  distribution from $\mathbf{N}(\mathcal{D})$.  Then for any two sets
  $X,Y\subseteq V$ with $|X|=|Y|$, it holds that
  \[
  \det\left(\Sigma_{X, Y}\right) = 
  \sum \;(-1)^\Pi
  \prod_{\tau\in\Pi} \sigma(\tau),
  \]
  where the sum is over systems of treks $\Pi: X\rightrightarrows Y$
  without sided intersection.  In particular, there is a system of
  treks from $X$ to $Y$ in $\mathcal{D}$ without sided intersection if
  and only if
  \[
  \det\left(\Sigma_{X, Y}\right) \not= 0
  \]
  for the covariance matrix $\Sigma$ of some distribution
  in $\mathbf{N}(\mathcal{D})$.
\end{lemma}

The sign of $\det(\Sigma_{X,Y})$ in Lemma~\ref{lem:std} is only
well-defined after an ordering has been established for the elements
of $X$ and $Y$.  Given such orderings, the sign $(-1)^\Pi$ of a trek
system $\Pi$ is defined in terms of the permutation arising from the
bijection between $X$ and $Y$ that maps the sources of the treks in
$\Pi$ to their targets.  The details are irrelevant for the subsequent
use of Lemma~\ref{lem:std}.

The final concept to be introduced is d-connection; see
e.g.~\cite{lauritzen:1996}.  Let $u\not=v$ be distinct nodes, and let
$A\subseteq V$.  A walk $\pi$ from $u$ to $v$ is \emph{d-connecting
  given $A$} if
\begin{enumerate}
\item every collider in $\pi$ is in $A$, and
\item\label{dconn-walk-2} every non-collider in $\pi$ is not in $A$.
\end{enumerate}
Condition~(\ref{dconn-walk-2}) implies that $u,v\notin A$, as only
interior nodes on $\pi$ can be colliders.  Note that treks from $u$ to
$v$ are precisely the d-connecting walks for evidence set
$A=\emptyset$.  The next lemma makes the connection between
d-connecting walks and non-zero conditional covariances.

\begin{lemma}
  \label{lem:dconnect}
  Let $A\subseteq V$ and $u,v\in V\setminus A$.  Then there is a
  d-connecting walk from $u$ to $v$ given $A$ if and only if there
  exists a distribution in $\mathbf{N}(\mathcal{D})$ whose covariance
  matrix $\Sigma=(\sigma_{uv})$ has
  \[
  \sigma_{uv.A} \;:=\; \sigma_{uv} -
  \Sigma_{u,A}\left(\Sigma_{A,A}\right)^{-1} \Sigma_{A,v} 
  %% \;=\;
  %% \frac{\det\left(\Sigma_{uA,vA}\right)}{\det\left(\Sigma_{A,A}\right)} 
  \;\not=\;   0.
  \]
\end{lemma}

For two nodes $u\not=v$, we define the \emph{top} node of a
d-connecting walk $\pi$, denoted $\tp{\pi}$, as the top node of the
(non-trivial) trek starting at the source node $u$ and ending at the
first node in $A\cup\{v\}$ that is visited by $\pi$.  If $\pi$ has no
colliders and is thus itself a trek, this definition of $\tp{\pi}$ is
consistent with the definition for the top node of a trek.  Let
$\mathcal{W}_A(u,v)$ be the set of all walks from $u$ to $v$ that are
d-connecting given $A$.  Then we write
\begin{equation}
\label{eq:topnodes}
\tops_A(u,v) = \bigcup_{\pi\in\mathcal{W}_A{(u,v)}} {\tp{\pi}}
\end{equation}
for the set of all top nodes in walks from $u$ to $v$ that are
d-connecting given $A$.

\begin{example} Consider again the graph from
  Figure~\ref{fig:trek_examples}.
  
  \begin{itemize}
  \item[(a)] A system of treks $\Pi=\{\tau_1,\tau_2\}$ from
    $X=\{4,5\}$ to $Y=\{3,4\}$ is given by:
    \begin{align*}
      \tau_1: 4\leftarrow 1\to 3 \to 4, && \tau_2: 5\leftarrow 2 \to 3.
    \end{align*}
    The system $\Pi$ has a sided intersection because
    $3\in\rhs{\tau_1} \cap \rhs{\tau_2}$.
    
  \item[(b)] The set $\Pi=\{\tau_1,\tau_2\}$ comprising the two treks
    \begin{align*}
      \tau_1: 3\leftarrow 1\to 4 \to 5, && \tau_2: 5\leftarrow 2
    \end{align*}
    is a system of treks from $X=\{3,5\}$ to $Y=\{2,5\}$ that has no
    sided intersection.  The node 5 appears on different sides in
    $\tau_1$ and in $\tau_2$.

  \item[(c)] Let $A=\{3,5\}$.  The walk
    \begin{align*}
      \pi: 2 \rightarrow 3 \leftarrow 1 \to 4
    \end{align*}
    from 2 to 4 is d-connecting given $A$, with $\tp{\pi} =
    2$.  Since all walks from 2 to 4 start with edge $2\to 3$ or
    $2\to 5$, we have $\tops_A(2,4)=\{2\}$.
    %% as the only other d-connecting walk
    %% given $A$ is $2\to 5\leftarrow 4$.
    
  \item[(d)] Let $A=\{5\}$.  Then the walk
    \begin{align*}
       \pi: 2 \rightarrow 3 \leftarrow 1 \to 4
    \end{align*}
    is not a d-connecting walk given $A$.  However,
    \begin{align*}
      \pi': 2 \to 3\to 4\to 5 \leftarrow 4
    \end{align*}
    is d-connecting given $A$.  For the same reason as in (c),
    $\tops_A(2,4)=\{2\}$.
  \end{itemize}
\end{example}

The trek rule from Lemma~\ref{lem:trek-rule}, the description of
determinants in terms of trek systems from Lemma~\ref{lem:std}, and
the result on d-connecting walks from Lemma~\ref{lem:dconnect} each
have generalizations to mixed graphs.  In these generalizations, the
notion of a collider is extended to also include vertices $w$ for
which the incident edges are of the form $w'\to w\bi w''$, $w'\bi
w\leftarrow w''$, or $w'\bi w\bi w''$; details can be found
in~\cite{richardson:2002,sullivant:2010}.  In the special case of
chain graphs (and thus also for acyclic digraphs), it holds in
addition that the model can be described entirely by conditional
independence constraints \cite{richardson:2002,wermuth:2004}.

\begin{lemma}
  \label{lem:ci}
  If $\mathcal{G}=(V,D,B)$ is a chain graph, then a positive definite
  $V\times V$ matrix $\Sigma=(\sigma_{uv})$ is the covariance matrix
  of a distribution in $\mathbf{N}(\mathcal{G})$ if and only if
  \[
  \sigma_{uv.A} \;:=\; \sigma_{uv} -
  \Sigma_{u,A}\left(\Sigma_{A,A}\right)^{-1} \Sigma_{A,v} =0
  \]
  for all $A\subseteq V$ and $u,v\in V\setminus A$ for which there
  does not exist a d-connecting walk from $u$ to $v$ given $A$.
\end{lemma}

Finally, the uniqueness results from~(\ref{eq:regress-coeff})
and~(\ref{eq:cond-var}) continue to hold for chain graphs but may fail
for more general mixed graphs; see, for instance, the discussion in
the introduction of \cite{drtoneichler:2009}.

%%%%%%%%%%%%%%%%%%%%%%%%%%%%%%%%%%%%%%%%%%%%%%%%%
\section{Chain Graphs}
\label{sec:chaingraphs}

In this section we prove the necessity of the condition from
Theorem~\ref{thm:main}.  So suppose, throughout this section, that
$\mathcal{G}=(V,D,B)$ is a chain graph.  Our starting point is
information about the structure of the covariance matrices in
$\mathbf{N}(\mathcal{G})$.

Let $C\subseteq V$ be a chain component, that is, a connected
component of the bidirected part $(V,B)$.  Let
\[
A=\an(C)=\bigcup_{c\in C} \an(c)
\]
be the ancestors of the nodes in $C$.  The set of ancestors of a node $v$, $\an(v)$, is the set of all nodes $u$ such that a directed path from $u$ to $v$ exists. Note that this path is allowed to be trivial; i.e., $v\in\an(v)$. We write $wA\equiv \{w\}\cup A$ when $w\in V$ and define $B_C:=B\cap(C\times C)$ to be the set of
edges between nodes in $C$.  Then the following fact is well-known
from the characterization of $\mathbf{N}(\mathcal{G})$ in terms of
conditional independence that we stated as Lemma~\ref{lem:ci}.

\begin{lemma}
  \label{lem:zero-schur-complement}
  Let $u,v\in C$ be two nodes in the chain component $C\subseteq V$ of
  the chain graph $\mathcal{G}=(V,D,B)$.  Then the nodes $u$ and $v$
  are non-adjacent if and only if
  \[
  \sigma_{uv.A} \;:=\; \sigma_{uv} -
  \Sigma_{u,A}\left(\Sigma_{A,A}\right)^{-1} \Sigma_{A,v} 
  \;=\;
  \frac{\det\left(\Sigma_{uA,vA}\right)}{\det\left(\Sigma_{A,A}\right)} 
  \;=\;   0
  \]
  for all covariance matrices $\Sigma=(\sigma_{ij})$ of distributions
  in $\mathbf{N}(\mathcal{G})$.  Moreover, every matrix in
  $\mathit{PD}(B_C)\subset\mathbb{R}^{C\times C}$ is the
  conditional covariance matrix
  \[
  \Sigma_{C.A}=(\sigma_{uv.A})_{u,v\in C}
  \]
  of a distribution in $\mathbf{N}(\mathcal{G})$.
\end{lemma}

In the sequel, suppose that the bidirected part of the chain graph
$\mathcal{G}=(V,D,B)$ is not decomposable, that is, there is a chain
component $C\subseteq V$ whose induced bidirected subgraph contains a
chordless cycle of length $p\ge 4$.  For notational convenience, we
label the nodes on this cycle by $[p]:=\{1,\dots,p\}$ in such a way
that adjacent nodes have sequential labels modulo $p$.  In other
words, $u,v\in [p]$ are adjacent if and only if $|u-v|\in\{1,p-1\}$.
Note that $A\subseteq V\setminus[p]$.

Let $B_p=B\cap([p]\times [p])$ be the set of edges in the considered
bidirected $p$-cycle.  And for an acyclic digraph $\mathcal{D}=(U,E)$
with $U\supseteq V$, let $\mathit{PD}(\mathcal{D})$ be the set of
covariance matrices of distributions in $\mathbf{N}(\mathcal{D})$, and
let
\begin{equation}
  \label{eq:all-schur-D}
  \mathit{PD}_{[p]}(\mathcal{D}|A) = \left\{
    \Sigma_{[p].A} \::\: \Sigma\in\mathit{PD}(\mathcal{D})
  \right\}
\end{equation}
be the associated set of $[p]\times [p]$ Schur complements/conditional
covariance matrices given $A$.  Here,
$\Sigma_{[p].A}=(\sigma_{uv.A})_{u,v\in[p]}$ as in
Lemma~\ref{lem:zero-schur-complement}.

Using Lemma~\ref{lem:zero-schur-complement}, with the adopted labeling
convention, we see that Theorem~\ref{thm:main} is implied by the
following fact.

\begin{proposition}
  \label{prop:chain-graph}
  There does not exist an acyclic digraph $\mathcal{D} = (U,E)$ on
  $U\supseteq V$ such that
  $\mathit{PD}_{[p]}(\mathcal{D}|A)=\mathit{PD}(B_p)$.
\end{proposition}

Our approach is based on a sign-change trick from \cite{drtonyu:2010}.
For a matrix $\Phi=(\phi_{uv})\in\mathbb{R}^{p\times p}$, let
$\Phi^{(12)} = (\phi^{(12)}_{uv})$ be the $p \times p$ matrix that
coincides with $\Phi$ except for the $(1,2)$ and $(2,1)$ entries,
which are negated.  Hence,
\begin{align}
  \label{eq:sign-negation}
  \phi^{(12)}_{uv}  &= 
  \begin{cases}
    -\phi_{uv} &\text{if}\quad (u,v)\in\{(1,2), (2,1)\},\\
    \phi_{uv} &\text{if}\quad (u,v)\notin\{(1,2), (2,1)\}.
  \end{cases}
\end{align}
By \cite[Example 5.2]{drtonyu:2010}, there are matrices
$\Phi\in\mathit{PD}(B_p)$ for which $\Phi^{(12)}$ is not positive
definite, so that $\Phi^{(12)}\notin\mathit{PD}(B_p)$.  It follows
that Proposition~\ref{prop:chain-graph} is an implication of the next
fact.

\begin{proposition}
  \label{prop:chain-graph-negation}
  Let $\mathcal{D}=(U,E)$ be an acyclic digraph on $U\supseteq V$ such
  that $\mathbf{N}_V(\mathcal{D})$ is a full-dimensional subset of
  $\mathbf{N}(\mathcal{G})$.  Then
  $\Phi\in\mathit{PD}_{[p]}(\mathcal{D}|A)$ implies that
  $\Phi^{(12)}\in\mathit{PD}_{[p]}(\mathcal{D}|A)$.
\end{proposition}

\begin{proof}[Proof of Proposition~\ref{prop:chain-graph-negation}]
  Let $u,v$ be two distinct nodes in $[p]$.  By
  Lemma~\ref{lem:zero-schur-complement} and Lemma~\ref{lem:dconnect}, since
  $\mathbf{N}_V(\mathcal{D})$ is a full-dimensional subset of
  $\mathbf{N}(\mathcal{G})$, the graph $\mathcal{D}$ contains a d-connecting walk
  from $u$ to $v$ given the ancestors in $A$ if and only
  if $(u,v)\in B_p$.  Similarly, by Lemma~\ref{lem:zero-schur-complement} and
  Lemma~\ref{lem:std}, the graph $\mathcal{D}$ contains a system of
  treks $\Pi:\{u\}\cup A\rightrightarrows \{v\}\cup A$ without sided
  intersection if and only if $(u,v)\in B_p$.
  
  Now write $\Sigma$, the covariance matrix of a distribution in
  $\mathbf{N}(\mathcal{D})$, as
  \begin{equation}
    \label{eq:represent-sigma}
  \Sigma= (I-\Gamma)^{-T}\Delta(I-\Gamma)^{-1}
  \end{equation}
  with $\Gamma=(\gamma_{uv})\in\mathbb{R}^E$ and a diagonal matrix
  $\Delta\in\mathbb{R}^{U\times U}$ that has positive diagonal
  entries.  Suppose
  \[
  \Phi = \Sigma_{[p].A}
  \]
  is the associated conditional covariance matrix.  We will show how
  to define a matrix $\Gamma'=(\gamma_{uv}')\in\mathbb{R}^E$ such that 
  \begin{equation}
    \label{eq:represent-sigma-negated}
  \Sigma'=(\sigma'_{uv})=
  (I-\Gamma')^{-T}\Delta(I-\Gamma')^{-1}
  \end{equation}
  has conditional covariance matrix 
  \begin{equation}
    \label{eq:chain-graph-negation-to-show}
   {\Sigma'}_{[p].A} =\Phi^{(12)}.
  \end{equation}
  
  Consider the set $\tops_A(1,2)$ of all top nodes of walks from $1$
  to $2$ in $\mathcal{D}$ that are d-connecting given $A$;
  recall~(\ref{eq:topnodes}).  Note that $\tops_A(1,2)\cap
  A=\emptyset$.  For a node $u\notin A$, let $\an_{U\setminus A}(u)$
  be the set of ancestors $v\in\an(u)$ for which there is a directed
  path from $v$ to $u$ that is d-connecting given $A$.  In other
  words, this directed path does not contain any nodes in $A$.  We
  allow trivial d-connecting paths from a node outside $A$ to itself
  so that $u\in\an_{U\setminus A}(u)$.  Now, define
  \begin{equation}
    \label{eq:E12-chain-graph}
    E_A(1,2) = \{u \to v\in E \::\: u\in\tops_A(1,2),\, v\in
    \an_{U\setminus A}(1)\setminus\tops_A(1,2)\}
  \end{equation}
  to be the set of all edges that originate at a top node of a
  d-connecting walk and point to a node that is an ancestor of $1$
  along a directed path outside of $A$ but not a top node itself.  (We
  illustrate this definition in Example~\ref{ex:EA12} below.)
  %% It is not difficult to see that an edge
  %% $u\to v\in E$ is in $E(1,2)$ if and only if there is a trek
  %% $\tau\in\mathcal{T}(1,2)$ with $u=\tp{\tau}$ and $v\in\lhs{\tau}$
  %% but no trek $\tau\in\mathcal{T}(1,2)$ has $v\in\rhs{\tau}$.  The
  %% former trek $\tau$ can be chosen to use edge $u\to v$.
  
  Let $\Gamma=(\gamma_{uv})\in\mathbb{R}^E$ be the matrix
  from~(\ref{eq:represent-sigma}).  Define the matrix
  $\Gamma'=(\gamma_{uv}')\in\mathbb{R}^E$ by setting
  \begin{align}
    \label{eq:edge-negation}
    \gamma'_{uv}  &= \begin{cases}
      -\gamma_{uv} &\text{if}\quad u\to v \in E_A(1,2),\\
      \gamma_{uv} &\text{if}\quad u\to v \notin E_A(1,2).
    \end{cases}
  \end{align}
  We claim that this choice of $\Gamma'$ satisfies the desired
  equality from~(\ref{eq:chain-graph-negation-to-show}).  
  
  As stated in Lemma~\ref{lem:zero-schur-complement}, conditional
  covariances are a ratio of two determinants.  In the rest of this
  section, we derive a series of lemmas that lead to
  Corollary~\ref{cor:AtoA}, according to which
  %%,~\ref{cor:1Ato2A} and~\ref{cor:iAtojA}
  $\det(\Sigma'_{1A,2A})=-\det(\Sigma_{1A,2A})$ while all other
  concerned determinants remain unchanged when replacing $\Sigma$ by
  $\Sigma'$.  It follows that (\ref{eq:chain-graph-negation-to-show})
  indeed holds for our choice of $\Gamma'$.
\end{proof}

\begin{figure}[t]
  \centering
  (a) \hspace{-0cm}
  \includegraphics[scale=0.25]{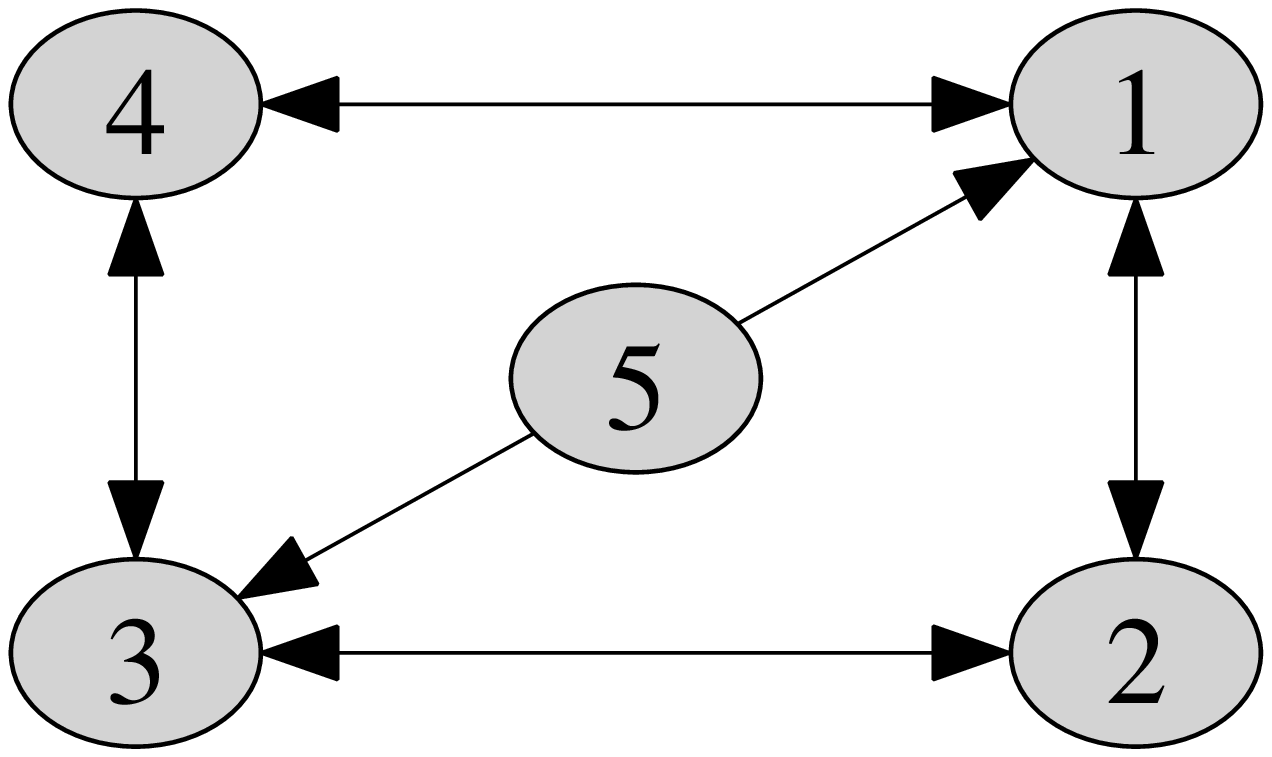} 
  \hspace{0.5cm}
  (b) \hspace{-0cm}
  \includegraphics[scale=0.25]{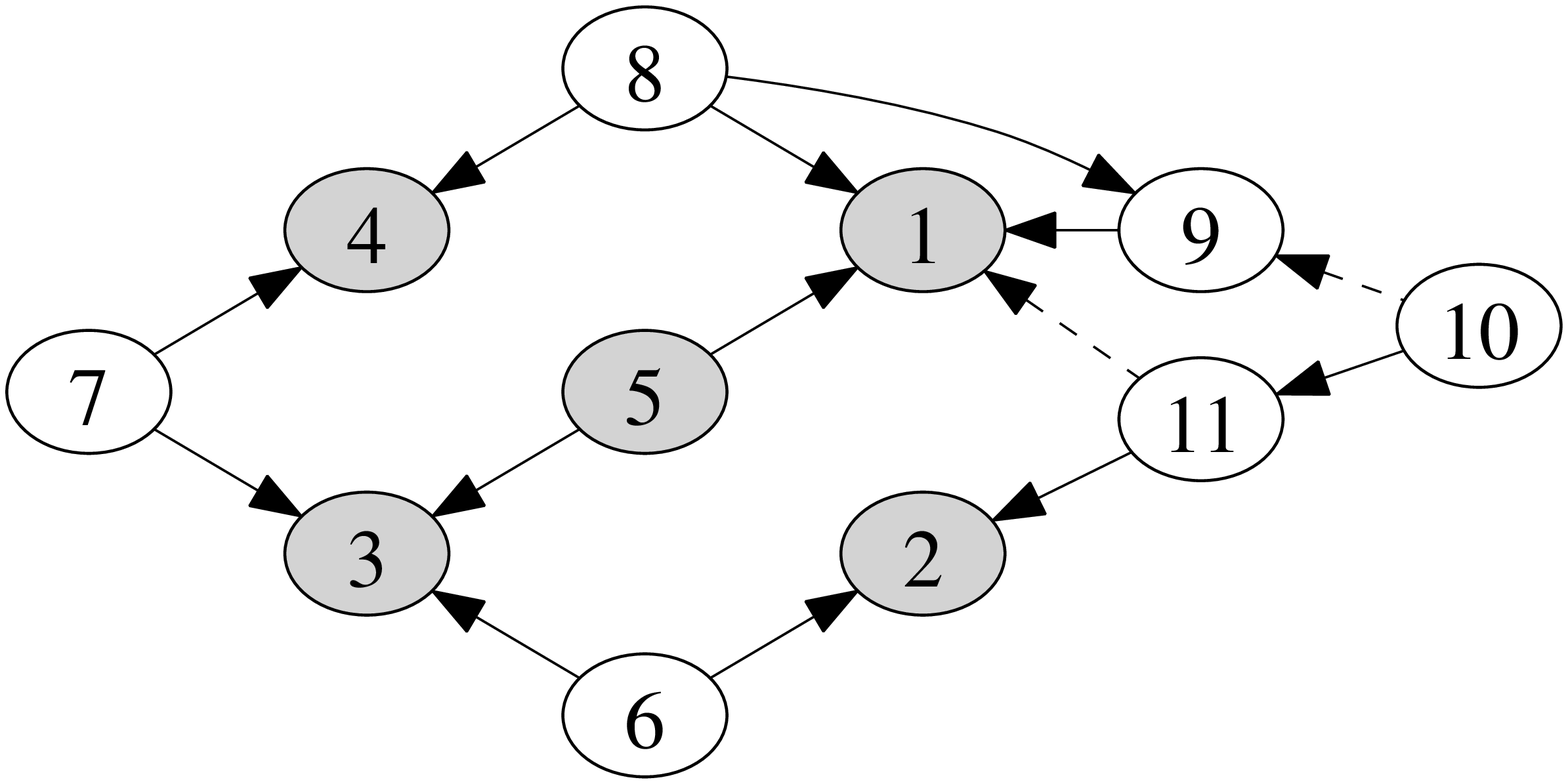}
  \caption{(a) A chain graph $\mathcal{G}$ 
    with vertex set $V=\{1,\ldots,5\}$ that has non-decomposable
    bidirected part.  (b) An acyclic digraph $\mathcal{D}$ such that
    $\mathbf{N}_V(\mathcal{D})$ is a full-dimensional subset of
    $\mathbf{N}(\mathcal{G})$; dashed arrows indicate edges contained
    in $E_A(1,2)$.}
  \label{fig:negation}
\end{figure}

\begin{example}
  \label{ex:EA12}
  Let $\mathcal{G}=(V,D,B)$ be the chain graph 
  %% on $V=\{1,\ldots, 5\}$
  depicted in Figure \ref{fig:negation}(a).  The acyclic digraph
  $\mathcal{D} = (U,E)$ shown in Figure \ref{fig:negation}(b) is such
  that $\mathbf{N}_V(\mathcal{D})$ is a full-dimensional subset of
  $\mathbf{N}(\mathcal{G})$.  The graph $\mathcal{G}$ has the
  non-decomposable chain component $C = \{1,\ldots, 4\}$, with
  $A=\an(C)=\{5\}$.  In the digraph $\mathcal{D}$, we have
  $\tops_A(1,2) = \{10,11\}$ and $E_A(1,2) = \{10\rightarrow 9,
  11\rightarrow 1\}$.  The two edges in $E_A(1,2)$ are represented as
  dashed arrows in Figure \ref{fig:negation}(b).  The corresponding
  entries in $\Gamma\in\mathbb{R}^E$ are negated in the construction
  of $\Gamma'$ in (\ref{eq:edge-negation}).
\end{example}

Let $\mathcal{D}_{U\setminus A}$ be the subgraph induced by
$U\setminus A$.  For a set $Z\subseteq U\setminus A$, we let
\[
\an_{U\setminus A}(Z) = \bigcup_{u\in Z} \an_{U\setminus A}(u)
\]
be the ancestors of $Z$ in the induced subgraph
$\mathcal{D}_{U\setminus A}$.  According to the next lemma, the set of
$\tops_A(1,2)$ is ancestral in this induced subgraph.  Throughout the
rest of the section, d-connecting walks are always d-connecting given
$A$.

\begin{lemma}
  \label{lem:tops-ancestral-notA}
  The top nodes of d-connecting walks from $1$ to $2$ form an
  ancestral subset of $\mathcal{D}_{U\setminus A}$, that is,
  \[
  \an_{U\setminus A}\big(\tops_A(1,2)\big)\subseteq \tops_A(1,2).
  \]
\end{lemma}
\begin{proof}
  As noted above $\tops_A(1,2)\cap A=\emptyset$.  Now suppose that $v\in\tops_A(1,2)$ and there exists 
  a directed path from $u$ to $v$ in $\mathcal{D}_{U\setminus A}$. Choose a d-connecting walk $\pi$ from $1$ to
  $2$ with $\tp{\pi}=v$.  At $v$, we may insert into $\pi$, the trek
  \[
  v \leftarrow \cdots \leftarrow u \to \cdots \to v
  \]
  that uses twice the directed path from $u$ to $v$ that exists in
  $\mathcal{D}_{U\setminus A}$.  The insertion yields a walk $\pi'$
  from $1$ to $2$ that is d-connecting and has $\tp{\pi'}=u$.
\end{proof}

Let $u,v\in[p]$.  Using Lemma~\ref{lem:std}, we may write the
determinants of $\Sigma_{uA,vA}$ and $\Sigma_{A,A}$ for $\Sigma$
from~(\ref{eq:represent-sigma}) and the analogous determinants for
$\Sigma'$ from~(\ref{eq:represent-sigma-negated}) as sums over systems
of treks without sided intersection.  The treks in these systems are
thus treks in $\mathcal{D}$ that are formed from two directed paths
that do not contain an edge $w\to x$ with source node $w\in A$.  We
refer to such paths as \emph{proper directed paths}.  If a proper
directed path ends with a target node not in $A$ then it is a directed path
in the induced subgraph $\mathcal{D}_{U\setminus A}$.

%------------------- LEMMA 3 BELOW -------------

\begin{lemma}
  \label{thm:targetA}
  A proper directed path with a target in $A$ cannot have an
  edge in $E_A(1,2)$. 
\end{lemma}

\begin{proof}
  Let $\alpha\in A$, and suppose for contradiction that there exists a
  proper directed path $\pi$ from a node $u$ to $\alpha$ that contains
  an edge $v \to w \in E_A(1,2)$.  Hence,
  \[
  \pi: u \to \cdots \to v \to w \to \cdots \to \alpha.
  \]
  By definition of $E_A(1,2)$, we have $w \in \an_{U\setminus A}(1)$.
  Thus, there exists a trek $\tau$ in $\mathcal{D}$ of the form
  \[
  \tau: 1 \leftarrow \cdots \leftarrow w \rightarrow \cdots
  \rightarrow \alpha
  \]
  that does not have any interior nodes in $A$.
  
  Since $v \in \tops_A(1,2)$ by definition of $E_A(1,2)$, it follows
  from Lemma~\ref{lem:tops-ancestral-notA} that $u \in \tops_A(1,2)$.
  Thus there exists a d-connecting walk $\delta$ from 1 to 2 with
  $u = \tp{\delta}$.  Define $\delta'$ to be the subwalk of $\delta$
  from $u$ to 2 formed by removing the directed walk from $u$ to 1 at
  the beginning of $\delta$.  Concatenating $\tau$, the path $\pi$
  reversed, and $\delta'$ yields a d-connecting walk $\pi'$ from 1 to 2 of the form
  \[
  \pi': 1 \leftarrow \cdots \leftarrow w \rightarrow \cdots
  \rightarrow \alpha \leftarrow \cdots \leftarrow u \; \cdots\;2
  \]
  that has $w=\tp{\mathcal{\pi}'}$.
  Consequently, $w \in \tops_A(1,2)$, contradicting the assumption
  that $v \to w \in E_A(1,2)$.
\end{proof}

%------------------- LEMMA 4 BELOW -------------

\begin{lemma}
  \label{thm:targetj}
  A proper directed path with a target in $\{2,\dots,p\}$ cannot have an
  edge in $E_A(1,2)$. 
\end{lemma}

\begin{proof}
  Let $x\in \{2,\dots,p\}$, and suppose for contradiction that $v \to
  w \in E_A(1,2)$ is an edge in a proper directed path $\pi$ from a
  node $u$ to $x$.  Written in reverse, $\pi$ is
  \[
  x \leftarrow \cdots \leftarrow w \leftarrow v \leftarrow \cdots
  \leftarrow u. 
  \]
  We distinguish two cases based on whether $x=p$ or not.
  
  If $x\in\{2,\ldots,p-1\}$, then $\mathcal{D}_{U\setminus A}$
  contains a trek
  \[
  \tau: 1 \leftarrow \cdots \leftarrow w \rightarrow \cdots \rightarrow x,
  \]
  since $w\in\an_{U\setminus A}(1)$.  If $x=2$, then $\tau$ is
  d-connecting and, thus, $w\in \tops_A(1,2)$.  This contradicts the
  assumption that $v \to w \in E_A(1,2)$.  If $x\in\{3,\ldots ,p-1\}$,
  then $\tau$ is a d-connecting trek between two nodes that are not
  adjacent in the bidirected cycle $([p],B_p)$, which is again a
  contradiction.
  
  Finally, suppose that $x=p$.  Since $v \in \tops_A(1,2)$, we also
  have $u \in \tops_A(1,2)$ by the ancestrality property from
  Lemma~\ref{lem:tops-ancestral-notA}.  Hence, $\mathcal{D}$ contains
  a d-connecting walk $\delta$ from 1 to 2 such that $u =
  \tp{\delta}$.  Define $\delta'$ to be the subwalk of $\delta$ from
  $u$ to 2 formed by removing the directed path from $u$ to 1 at the
  beginning of $\delta$.  Concatenating $\pi$ and $\delta'$ yields a
  walk
  \[
  \pi': p \leftarrow \cdots \leftarrow w \leftarrow v \leftarrow
  \cdots \leftarrow u \; \cdots \; 2, 
  \]
  from $p$ to $2$ that is d-connecting.  This is a contradiction since
  $(2,p)\notin B_p$.
\end{proof}

%% The next lemma concerns systems of treks without sided intersection
%% from $\{1\}\cup A$ to $\{2\}\cup A$, which in the present setting have
%% the following special property.

%------------------- LEMMA 5 BELOW -------------

\begin{lemma}
  \label{thm:nonemptylhs}
  Let $\Pi:\{1\}\cup A\rightrightarrows \{2\}\cup A$ be a system of
  treks without sided intersection, and let $\tau_1$ be the trek in
  $\Pi$ with source 1.  Then $\lhs{\tau_1} \neq \{1\}$,
  $\tp{\tau_1}\neq 1$, and there is a proper directed path from
  $\tp{\tau_1}$ to 1.  Moreover, $\tp{\tau_1}\in\tops_A(1,2)$.
\end{lemma}

\begin{proof}
  Clearly, $1 \in \lhs{\tau_1}$.  Suppose for contradiction that
  $\lhs{\tau_1}=\{1\}$.  Then $\tp{\tau_1}=1$, implying that $\tau_1$
  is a directed path beginning at 1.  Appending to $\tau_1$ other
  treks in $\Pi$, we may form a d-connecting walk $\pi$ from
  $1$ to $2$ that begins with a directed edge $1\to \cdot\,$ pointing away
  from $1$.
  
  Since $p$ and $1$ are adjacent in the bidirected cycle $([p], B_p)$,
  the graph $\mathcal{D}$ contains a d-connecting walk $\pi'$ from $p$
  to 1.  Appending $\pi$ to $\pi'$ yields a d-connecting walk $\delta$
  from $p$ to 2, which is a contradiction because $(2,p)\not\in B_p$.
  We conclude that $\lhs{\tau_1} \neq \{1\}$, and the other claims are
  immediate consequences.  In particular,
  $\tp{\tau_1}=\tp{\pi}\in\tops_A(1,2)$.
\end{proof}

%------------------- LEMMA 6 BELOW -------------

\begin{lemma}
  \label{thm:Tto1}
  A proper directed path with target 1 and source in $\tops_A(1,2)$
  has exactly one edge in $E_A(1,2)$.
\end{lemma}

\begin{proof}
  First, note that $1\notin \tops_A(1,2)$.  Indeed, if
  $1\in\tops_A(1,2)$, then there is a d-connecting walk from $1$ to
  $2$ that begins with a directed edge pointing away from $1$.  The
  arguments in the proof of Lemma \ref{thm:nonemptylhs} then lead to a
  contradiction.
  
  Now, let $\pi$ be any proper directed path from a node $u\in
  \tops_A(1,2)$ to 1.  As we traverse $\pi$ from $u$ to 1, there must
  exist some first node, say $w$, that is not in $\tops_A(1,2)$.  Let
  $v$ be the node that immediately precedes $w$ on $\pi$, and note
  that $v\in \tops_A(1,2)$.   Then $v\to w\in E_A(1,2)$.
  
  By Lemma~\ref{lem:tops-ancestral-notA}, $\tops_A(1,2)$ is ancestral
  with respect to $\mathcal{D}_{U\setminus A}$.  Hence, no descendant
  of $w$ along $\pi$ is in $\tops_A(1,2)$, for else we would have
  $w\in\tops_A(1,2)$.  We conclude that, along $\pi$, every node from
  $u$ to $v$ is in $\tops_A(1,2)$, and no node from $w$ to 1 is in
  $\tops_A(1,2)$.  Hence, $v\to w$ is the only edge in $\pi$
  that is in $E_A(1,2)$.
\end{proof}

The preceding lemmas yield the following corollary about
subdeterminants of the matrices $\Sigma$ and $\Sigma'$
from~(\ref{eq:represent-sigma})
and~(\ref{eq:represent-sigma-negated}), respectively.

%------------------- COROLLARY 1 BELOW -------------

\begin{corollary}
  \label{cor:AtoA}
  The matrices $\Sigma$ and $\Sigma'$ from~(\ref{eq:represent-sigma})
  and~(\ref{eq:represent-sigma-negated}) satisfy
  \begin{enumerate}
  \item $\det(\Sigma'_{A,A})=\det(\Sigma_{A,A})$,
  \item $\det(\Sigma'_{1A,2A})=-\det(\Sigma_{1A,2A})$, and
  \item $\det(\Sigma'_{uA,vA})=\det(\Sigma_{uA,vA})$ for all pairs
  $(u,v)\in ([p]\times [p])\setminus \{(1,2),(2,1)\}$.
  \end{enumerate}
\end{corollary}

\begin{proof}
  In each case, we may appeal to Lemma~\ref{lem:std} and consider
  systems of treks without sided intersection.
  
  (i) Suppose $\Pi:A\rightrightarrows A$ is a system of treks without
  sided intersection.  Since there is no sided intersection, every
  trek in $\Pi$ must be the concatenation of two proper directed
  paths, each with target in $A$.  By Lemma~\ref{thm:targetA}, no trek
  in $\Pi$ contains an edge in $E_A(1,2)$.  Hence, no trek involves an
  edge $x\to y$ whose coefficient $\lambda_{xy}$ is negated when
  replacing $\Sigma$ by $\Sigma'$.

%------------------- COROLLARY 2 BELOW -------------

%% \begin{corollary}
%%   \label{cor:1Ato2A}
%%   $|\Sigma'_{1A,2A}|=-|\Sigma_{1A,2A}|$.
%% \end{corollary}
  
  (ii) Suppose $\Pi:\{1\} \cup A\rightrightarrows \{2\} \cup A$ is a
  system of treks without sided intersection.  Since there is no sided
  intersection, every trek in $\Pi$ must be the concatenation of two
  proper directed paths.  Exactly one of these proper directed paths
  has target 1 and, thus, contains exactly one edge in $E_A(1,2)$, by
  Lemma~\ref{thm:nonemptylhs} and Lemma~\ref{thm:Tto1}.  All other
  proper directed paths have target in $A$ or target 2.  By
  Lemma~\ref{thm:targetA} and Lemma~\ref{thm:targetj}, none of these
  paths contains an edge in $E_A(1,2)$.  We conclude that $\Pi$
  contains precisely one edge in $E_A(1,2)$.  Therefore, its
  contribution to the sum in Lemma~\ref{lem:std} is negated when
  replacing $\Sigma$ by $\Sigma'$.

%------------------- COROLLARY 3 BELOW -------------

%% \begin{corollary}
%%   \label{cor:iAtojA}
%%   $|\Sigma'_{iA,jA}|=|\Sigma_{iA,jA}|$ for all pairs
%%   $(i,j)\notin\{(1,2),(2,1)\}$.
%% \end{corollary}
  
  (iii) First, consider the case that $(u,v)\notin\{(1,p),(p,1)\}$.
  Then all relevant systems of treks without sided intersection are
  made up of proper directed paths with targets in $A$ or
  $\{2,\dots,p\}$.  By Lemma~\ref{thm:targetA} and
  Lemma~\ref{thm:targetj}, no edge from $E_A(1,2)$ appears in these
  paths and the claim follows.
  
  Finally, suppose that $(u,v)\in\{(1,p),(p,1)\}$, and let $\Pi:\{u\}
  \cup A\rightrightarrows \{v\} \cup A$ be a system of treks without
  sided intersection.  Splitting each trek in the system at its top
  node gives a collection of proper directed paths, of which exactly
  one path has target node 1.  If the source of this path, $x$, is in
  $\tops_A(1,2)$, then there is a d-connecting walk, $\delta$ from 1
  to 2 as well as a d-connecting walk $\delta'$ from 1 to $p$ such
  that $\tp{\delta}=\tp{\delta'}=x$.  Traversing $\delta'$ backwards
  from $p$ to $x$ and then traversing $\delta$ from $x$ to 2 traces
  out a d-connecting walk from $p$ to $2$, which is a contradiction
  since $p$ and $2$ are not adjacent in the bidirected cycle
  $([p],B_p)$.  Consequently, the proper directed path with target
  node 1 does not have its source in $\tops_A(1,2)$.  By the
  ancestrality from Lemma~\ref{lem:tops-ancestral-notA}, none of the
  edges in this path will be in $E_A(1,2)$.  Hence, no edge appearing
  in $\Pi$ has a coefficient that is negated when replacing $\Sigma$
  by $\Sigma'$.
\end{proof}

%% It follows directly from Corollary \ref{cor:AtoA}, Corollary
%% \ref{cor:1Ato2A}, Corollary \ref{cor:iAtojA}, and equation (\ref{eq:
%%   conditional-sigma}) that $\tilde{\Sigma}' = \tilde{\Sigma}^{(12)}$.
%% Hence $\tilde{\Sigma}^{(12)}\in\mathbf{N}_V(\mathcal{D})$, completing
%% the proof of Proposition \ref{prop:nondecomposable}.

%%%%%%%%%%%%%%%%%%%%%%%%%%%%%%%%%%%%%%%%%%%%%%%%%
\section{Discussion}
\label{sec:conclusion}

\begin{figure}[t]
  \centering
  \includegraphics[scale=0.25]{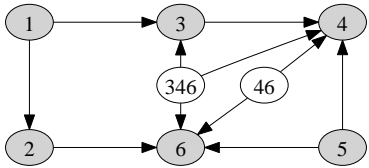} 
  \caption{A minimal Gaussian causal interpretation for the chain graph $\mathcal{G}$ in Figure \ref{fig:canonical}(a). The causality index of $\mathcal{G}$ is 2.}
  \label{fig:minimal}
\end{figure}

This paper provides a sufficient condition for a mixed graph and its associated Gaussian model to admit a strict causal interpretation in terms of acyclic digraphs with additional nodes that correspond to hidden Gaussian variables.   For chain graphs, we show the necessity of the condition.  An obvious follow-up problem would be to characterize strict Gaussian causality for more general classes of mixed graphs.  The ancestral graphs of \cite{richardson:2002} would be a natural starting point.

If a mixed graph $\mathcal{G}=(V,D,B)$ is strictly Gaussian causal, then there are infinitely many acyclic digraphs $\mathcal{D}=(U,E)$, $U\supseteq V$, that induce a hidden variable model $\mathbf{N}_V(\mathcal{D})=\mathbf{N}(\mathcal{G})$.  It would be interesting to study just how many hidden variables really need to be introduced for this equality of models.  To formalize the question, define the \emph{(Gaussian) causality index} of a mixed graph to be the minimum number $h$ such that there exists an acyclic digraph $\mathcal{D}$ on nodes $U\supseteq V$ with $|U|=|V|+h$ and $\mathbf{N}_V(\mathcal{D})=\mathbf{N}(\mathcal{G})$.  If $\mathcal{G}$ is not strictly Gaussian causal, set $h=\infty$.  The question is then whether we can efficiently determine the causality index of a general mixed graph.  To give an example, 
%% If $h$ is finite and $\mathcal{D}$ is an acyclic digraph that satisfies this condition, we call $\mathcal{D}$ a \emph{minimal Gaussian causal %% interpretation} for $\mathcal{G}$. For the chain graph $\mathcal{G}$ in Figure \ref{fig:canonical}(a), it can be shown that at least two hidden variables are needed. 
 the chain graph $\mathcal{G}$ in Figure \ref{fig:canonical}(a) has causality index two.  Figure \ref{fig:minimal} depicts an acyclic digraph $\mathcal{D}$ with two latent variables that induces a hidden variable model equal to the mixed graph model. 
For graphs $\mathcal{G}$ with causality index $\infty$ it would furthermore be interesting to determine inequalities that hold on hidden variable models $\mathbf{N}_V(\mathcal{D})$ that are full-dimensional subsets of $\mathbf{N}(\mathcal{G})$.  The `positive definiteness after sign change' that we use in this paper is one example of such an inequality.  For other work on inequality constraints, see \cite{evans:2012,kang:tian:2006} and references therein.

 %% Hence $\mathcal{D}$ is a minimal Gaussian causal interpretation for the chain graph $\mathcal{G}$ and it follows that the causality index of $\mathcal{G}$ is two. We would like to develop a straightforward procedure for determining the causality index of an arbitrary mixed graph, as well as a method for enumerating all its minimal Gaussian causal interpretations.

All of the above work considers the concrete setting of hidden variable models that are based on an assumption of joint multivariate normality of all variables, observed and hidden.  Alternatively, it would be interesting to treat a non-parametric version of the considered problem. To explain, let $\mathcal{G}=(V,D,B)$ be again a mixed graph with Gaussian model $\mathbf{N}(\mathcal{G})$.  For an acyclic digraph $\mathcal{D}=(U,E)$ with $U\supseteq V$ define $\mathbf{P}(\mathcal{D})$ to be set of all probability distributions on $\mathbb{R}^U$ that satisfy the conditional independence relations associated with $\mathcal{D}$; compare \cite{lauritzen:1996}.  Then define $\mathbf{P}_V(\mathcal{D})$ to be the set of all distributions on $\mathbb{R}^V$ that arise as $V$-marginal of a distribution in $\mathbf{P}(\mathcal{D})$.  Writing $\mathbf{N}(V)$ for the set of all multivariate normal distributions on $\mathbb{R}^V$ (with positive definite covariance matrix), we may then ask whether 
\begin{equation}
\label{eq:nonpara}
\mathbf{P}_V(\mathcal{D})\cap \mathbf{N}(V) \;= \; \mathbf{N}(\mathcal{G}).
\end{equation}
While the model $\mathbf{P}(\mathcal{D})$ is certainly much larger than its subset $\mathbf{N}(\mathcal{D})=\mathbf{P}(\mathcal{D})\cap \mathbf{N}(U)$, it is not clear to us that considering the model equality in~(\ref{eq:nonpara}) should give anything new.  In particular, the L\'evy-Cram\'er theorem \cite[\S8.8]{pollard:2002} suggests that when considering structural equations as in~(\ref{eq:sem}) one would have to consider non-linear equations in order to generate distributions in $\mathbf{P}_V(\mathcal{D})\cap \mathbf{N}(V)$ that are not already in $\mathbf{N}_V(\mathcal{D})$.  If it were the case that 
a strict causal interpretation cannot always be found via the addition of hidden non-Gaussian variables, then some Gaussian mixed graph models would be larger than needed for modeling of causally induced stochastic dependencies.

%\begin{itemize}
%\item For a chain graph $\mathcal{G}$ whose chain components are not
%  all decomposable, Theorem~\ref{thm:main} raises the question whether
 % or not there are non-linear or non-Gaussian directed graphical
 % models with hidden variables that can faithfully represent the
 % distributions that are not contained in faithful representations via
 % Gaussian hidden variable models.  If not mixed graph models are
 % larger than what is needed.
% \end{itemize}

\section*{Acknowledgments}

This work was supported by the NSF under Grant No.~DMS-0746265.
Mathias Drton was also supported by an Alfred P. Sloan Fellowship.

\bibliographystyle{amsalpha} 
\bibliography{mixed_equal}

\end{document}